\documentclass[a4paper,11pt]{amsart}
\usepackage{amsthm,amsmath,amssymb}
\usepackage{graphicx,subfigure}
\usepackage{hyperref}
\usepackage[utf8]{inputenc}
\usepackage{enumitem}
\usepackage{cmtt}

\def\ITEMMACRO #1 ??? #2 ???{\par\medskip\noindent%
\hangindent=#2em\setbox0\hbox{#1 \kern5pt}%
\ifdim\wd0<\hangindent\setbox0\hbox to\hangindent{\hss#1\quad}\fi%
\box0\ignorespaces}
\def\Item#1{\ITEMMACRO #1 ??? 2.5 ???}

\def\Bitem{\Item{\hss$\bullet$}}

\usepackage{graphicx}
\usepackage{color}
\graphicspath{{figs/}}

\renewenvironment{enumerate}{\begin{enumorig}[label=\textup{(\roman*)}, noitemsep, topsep=2pt plus 2pt, labelindent=.2em, leftmargin=*, widest=iii]}{\end{enumorig}}

\renewenvironment{itemize}{\begin{itemorig}[label=\textbullet, noitemsep, topsep=2pt plus 2pt, labelindent=.5em, labelsep=.5em, leftmargin=*]}{\end{itemorig}}

\usepackage{todonotes}

\newtheorem{theorem}{Theorem}

\newtheorem{proposition}[theorem]{Proposition}

\theoremstyle{remark}
\newtheorem*{claim}{Claim}

\newenvironment{claimproof}{%
\noindent%
\textit{Proof of Claim.}%
}
{
\hfill $\triangle$%
\medskip
}

\let\leq\leqslant
\let\geq\geqslant
\let\setminus\smallsetminus

\def\calC{\mathcal{C}}
\def\calL{\mathcal{L}}
\def\calG{\mathcal{G}}
\def\calF{\mathcal{F}}

\def\calR{\mathcal{R}}

\newcommand{\set}[1]{\left\{#1\right\}}

\makeatletter
\let\old@setaddresses\@setaddresses
\def\@setaddresses{\bgroup\parindent 0pt\let\scshape\relax\old@setaddresses\egroup}
\makeatother

\begin{document}

\title{On-line coloring between two lines}

\author[S.~Felsner]{Stefan Felsner}
\author[P.~Micek]{Piotr Micek}
\author[T.~Ueckerdt]{Torsten Ueckerdt}

\address[S.~Felsner]{Technische Universit\"at Berlin, Berlin, Germany}
\email{felsner@math.tu-berlin.de}

\address[P.~Micek]{Theoretical Computer Science Department, 
Faculty of Mathematics and Computer Science, Jagiellonian University, Poland}
\email{piotr.micek@tcs.uj.edu.pl}

\address[T.~Ueckerdt]{Department of Mathematics, 
Karlsruhe Institute of Technology, Germany}
\email{torsten.ueckerdt@kit.edu}

\thanks{P.\ Micek is supported by the Polish National Science Center 
within a grant UMO-2011/03/D/ST6/01370 and by the Polish Ministry of Science and Higher Education within the Mobility Plus program.}

\begin{abstract}
  We study on-line colorings of certain graphs given as intersection
  graphs of objects ``between two lines'', i.e., there is a pair of
  horizontal lines such that each object of the representation is a
  connected set contained in the strip between the lines and touches
  both.  Some of the graph classes admitting such a representation are
  permutation graphs (segments), interval graphs (axis-aligned
  rectangles), trapezoid graphs (trapezoids) and cocomparability
  graphs (simple curves). We present an on-line algorithm coloring
  graphs given by convex sets between two lines that uses
  $O(\omega^3)$ colors on graphs with maximum clique size
  $\omega$. 

  In contrast intersection graphs of segments attached to a single line 
  may force any on-line coloring algorithm to use an arbitrary number
  of colors even when $\omega=2$.

  The {\em left-of} relation makes the complement of intersection
  graphs of objects between two lines into a poset. As an aside we
  discuss the relation of the class $\calC$ of posets obtained from
  convex sets between two lines with some other classes of posets: all
  $2$-dimensional posets and all posets of height~$2$ are in $\calC$
  but there is a $3$-dimensional poset of height~$3$ that does not
  belong to $\calC$.

  We also show that the on-line coloring problem for
  curves between two lines is as hard as the on-line chain partition
  problem for arbitrary posets.
\end{abstract}

\maketitle

\section{Introduction}

In this paper we deal with on-line proper vertex coloring of graphs.
In this setting a graph is created vertex by vertex where each 
new vertex is created with all adjacencies to previously created vertices.
An \emph{on-line coloring algorithm} colors
each vertex when it is created, immediately and irrevocably, such that
adjacent vertices receive distinct colors.  In particular, when
coloring a vertex an algorithm has no information about future
vertices.  This means that the color of a vertex depends only on the
graph induced by vertices created before.  It is convenient to imagine
that vertices are created by some adaptive adversary so that 
the coloring process becomes a game between that adversary and an on-line
algorithm.

We are interested in on-line algorithms using a number of colors that
is bounded by a function of the chromatic number of the input graph.
For general graphs this is too much to ask for.  Indeed, it is a
popular exercise to devise a strategy for adversary forcing any
on-line algorithm to use arbitrarily many colors on a forest.
However, some restricted graph classes admit competitive on-line
coloring algorithms. Examples are $P_5$-free graphs~\cite{KPT95},
interval graphs~\cite{KT81} and cocomparability graphs~\cite{KPT94}.
All of these classes are covered by the main result of Penrice,
Kierstead and Trotter in~\cite{KPT94} that says that for any tree $T$
with radius $2$, the class of graphs that do not contain an induced
copy of $T$ can be colored on-line with the number of colors 
depending only on $T$ and the clique number of the input graph.

We are interested in situations where the on-line graph is presented
with a geometric intersection representation. A graph $G$ is an
\emph{intersection graph of a family $\mathcal F$ of sets} if the
vertices of $G$ and the elements of $\mathcal F$ are in bijection such
that two vertices are adjacent in $G$ if and only if the corresponding
sets intersect.  For convenience, we identify the intersection graph
of the family $\mathcal F$ with $\mathcal F$ itself.  The most
important geometric intersection graphs arise from considering
compact, arc-connected sets in the Euclidean plane $\mathbb{R}^2$.  In
the corresponding on-line coloring problem such objects are created
one at a time and an on-line coloring algorithm colors each set when
it is created in such a way that intersecting sets receive distinct
colors.  For many geometric objects the on-line coloring problem is
still hopeless, e.g., disks and axis-aligned squares (Erlebach and
Fiala~\cite{EF02}).  Since any intersection graph $G$ of translates of
a fixed compact convex set in the plane has a maximum degree bounded
by $6\omega(G)-7$ (Kim, Kostochka and Nakprasit~\cite{KKN04}), any
on-line algorithm that uses a new color only when it has to colors $G$
with at most $6\omega(G)-6$ colors.

In this paper we consider the on-line coloring problem of geometric
objects spanned between two horizontal lines, that is, arc-connected
sets that are completely contained in the strip $\mathbf{S}$ between
the two lines and have non-empty intersection with each of the lines.
Clearly, such a family imposes a partial order on its elements where
$x < y$ if $x$ and $y$ are disjoint and $x$ is contained in the left
component of $\mathbf{S}\setminus y$.  Hence, two sets intersect if
and only if they are incomparable in the partial order, i.e., the
intersection graph is a cocomparability graph. In particular, $\chi(G)
= \omega(G)$ for all such graphs $G$.  Conversely every
cocomparability graph has a representation as intersection graph of
$y$-monotone curves between two lines. The usual way to state this
result is by
saying that cocomparability graphs are function graphs, see~\cite{GRU83}
or~\cite{Lov83}. If the representation is given the cocomparability
graph comes with a transitive orientation of the complement. In this
setting there is an on-line algorithm that uses $\omega^{O(\log
  \omega)}$ colors when $\omega$ is the clique number of the graph
(Bosek and Krawczyk~\cite{BK10}, see also~\cite{BKKMS}). This subexponential function in
$\omega$ is way smaller than the superexponential function arising from the
on-line algorithm for cocomparability graphs from~\cite{KPT94}. 
The best known lower bound for on-line coloring of 
cocomparability graphs is of order $\Omega(\omega^2)$, see~\cite{BFKKMM-survey}. 
We present an on-line algorithm that uses only $O(\omega^3)$ colors
on convex objects spanned between two lines.

Intersection graphs of convex sets spanned between two lines
generalize several well-known graph classes.  

\Bitem {\em Permutation graphs} are
intersection graphs of segments spanned between two lines and posets
admitting such cocomparability graphs are the $2$-dimensional
posets.  

\Bitem {\em Interval graphs} are intersection graphs of axis-aligned
rectangles span\-ned between two horizontal lines.

\Bitem {\em Bounded tolerance graphs}  are intersection graphs of
parallelograms with two horizontal edges spanned between two
horizontal lines.\\
(Boun\-ded tolerance graphs were introduced in~\cite{GM82})

\Bitem {\em Triangle graphs} (a.k.a.~PI-graphs) 
are intersection graphs of triangles with a horizontal side spanned between
two horizontal lines. (Triangle graphs were introduced in~\cite{CK87})

\Bitem {\em Trapezoid graphs} are intersection graphs of trapezoids
with two horizontal edges spanned between two horizontal lines.
Posets admitting such cocomparability graphs are the posets of
interval-dimension at most $2$.  (Trapezoid graphs were
independently introduced in~\cite{CK87,DGP88}).
\medskip

\noindent
Effective on-line coloring algorithms have been known for some of
these classes: 

\Bitem {\em Permutation graphs} can be colored on-line with
$\binom{\omega+1}{2}$ colors (Schmerl 1979 unpublished,
see~\cite{BFKKMM-survey}).  Kierstead, McNulty and
Trotter~\cite{KMT84} generalized Schmerl's idea and gave an on-line
algorithm chain partitioning $d$-dimensional posets, presented with
$d$ linear extensions witnessing the dimension, and using
$\binom{\omega+1}{2}^{d-1}$ chains, here $\omega$ is the width of the
poset.

\Bitem {\em Interval graphs} can be colored on-line with
$3\omega-2$ colors (Kierstead and Trotter~\cite{KT81}).
\medskip

\noindent
An easy strategy for on-line coloring is given by \emph{First-Fit}, which is the strategy that colors each incoming vertex with the least admissible natural number.
While First-Fit uses $O(\omega)$ colors on interval orders (see~\cite{KST}) it is easy to trick this strategy and force arbitrary number of colors on permutation graph of clique-size $2$ (see~survey~\cite{BFKKMM-survey}).
The behavior of the First-Fit algorithm on $p$-tolerance graphs ($0<p<1$), a subclass of bounded tolerance graphs, was studied in~\cite{KS11}.
First-Fit uses there $O(\frac{\omega}{1-p})$ colors.

\begin{theorem}\label{thm:main-convex}
  There is an on-line algorithm coloring convex sets spanned between
  two lines with $O(\omega^3)$
  colors when  $\omega$ is the clique number of the intersection graph.
\end{theorem}

Proofs are deferred to later sections.

\begin{figure}[htb]
 \centering
 \includegraphics{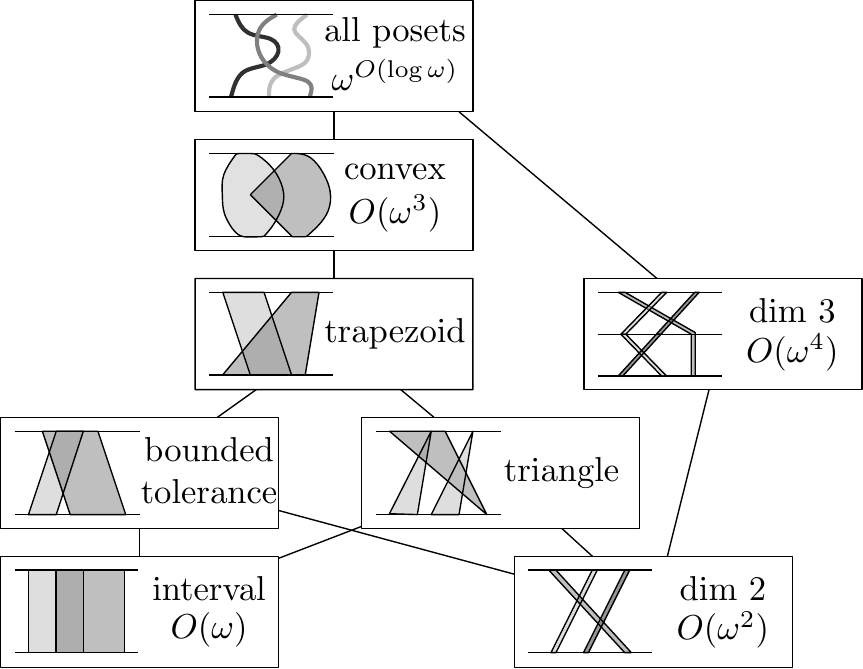}
 \caption{The containment order of some classes of graphs given by objects
  between two lines together with
  the performance guarantee for the best known on-line coloring
  algorithm (given the intersection representation as input).}
  \label{fig:poset-classes}
\end{figure}

\medskip\noindent
A poset is called \emph{convex} if its cocomparability graph is 
an intersection graph of convex sets spanned between two lines.
We give a short proof that all height $2$ posets are convex.
All $2$-dimensional posets are convex but not all $3$-dimensional.

\begin{proposition}\label{prop:convex-vs-dim-and-height}\hfill
\begin{enumerate}
\item Every height $2$ poset is convex;
\item There is a $3$-dimensional height $3$ poset that is not convex.
\end{enumerate}
\end{proposition}

Rok and Walczak~\cite{RW14} have looked at intersection graphs of
connected objects that are attached to a horizontal line and 
contained in the upper halfplane defined by this line. They show that 
there is a function $f$ such that  $\chi(G) \leq
f(\omega(G))$ for all $G$ admitting such a representation.
However, there is no effective on-line coloring algorithm 
for graphs in this class, even if we restrict the
objects to be segments.

\begin{proposition}\label{prop:attached-impossible}
  Any on-line algorithm can be forced to use arbitrarily many colors
  on a family of segments attached to a line, even if the family
  contains no three pairwise intersecting segments ($\omega = 2$).
\end{proposition}

Recall that it may make a difference for an on-line
coloring algorithm whether the input is an abstract
cocomparability graph, or the corresponding poset, or a geometric
representation.  Kierstead, Penrice and Trotter~\cite{KPT94} gave an
on-line coloring algorithm for cocomparability graphs using a number
of colors that is superexponential in $\omega$.  Bosek and
Krawczyk~\cite{BK10} introduced an on-line coloring algorithm for
posets using $\omega^{O(\log \omega)}$ colors where $\omega$ is the
width of the poset.  We show that having a poset represented
by $y$-monotone curves between two lines does not help on-line
algorithms. Indeed, such a representation can
be constructed on-line if the poset is given.

\begin{theorem}\label{thm:online-curves}
  There is an on-line algorithm that for any poset draws $y$-monotone
  curves spanned between two lines such that $x < y$ in the poset if and
  only if the curves $x$ and $y$ are disjoint and $x$ lies left of
  $y$.  That means, for every element of the poset when it is created
  a curve is drawn in such a way that throughout the set of already
  drawn curves forms a representation of the current poset.
\end{theorem}

Theorem~\ref{thm:main-convex} and
Proposition~\ref{prop:convex-vs-dim-and-height} are proven in
Section~\ref{sec:convex}. Actually we define the class of quasi-convex posets and
show the $O(\omega^3)$ bound for this class. Since every convex poset is
quasi-convex this implies Theorem~\ref{thm:main-convex}. The section
is concluded with a proposition showing that the class of quasi-convex
posets is a proper superclass of convex posets. In
Section~\ref{sec:connected-btw-two-lines} we discuss general connected
sets between two lines. In this context we prove
Theorem~\ref{thm:online-curves}.
We conclude the paper with a proof of
Proposition~\ref{prop:attached-impossible} in
Section~\ref{sec:connected-btw-two-lines} and a list of
four open problems related to these topics that we would very much
like to see answered.

\section{Quasi-Convex Sets Between Two Lines}\label{sec:convex}

A connected set $v$ spanned between two parallel lines 
is \emph{quasi-convex} if it contains a segment $s_v$ that has
its endpoints on the two lines. When working with a family of
quasi-convex sets it is convenient to fix such a segment $s_v$
for each $v$ and call it the \emph{base segment} of the set $v$.
Clearly, every convex set spanned
between two lines is also quasi-convex.

Below we show that there is an on-line algorithm coloring a family of
quasi-convex sets between two parallel lines with $O(\omega^3)$
colors, when $\omega$ is the clique number of the family.  This
implies Theorem~\ref{thm:main-convex}

\begin{proof}[Proof of Theorem~\ref{thm:main-convex}]
  We describe an on-line coloring algorithm using at most
  $\binom{\omega+1}{2}\cdot 24\omega$ colors on quasi-convex sets spanned
  between two parallel lines with clique number at most $\omega$.
  The algorithm colors incoming sets with triples
  $(\alpha,\beta,\gamma)$ of positive integers with $\alpha+\beta\leq
  \omega+1$ and $\gamma\leq 24\omega$ in such a way that intersecting
  sets receive different triples.

Let $\ell^1,\ell^2$ be the two horizontal lines such that the quasi-convex sets
of the input are spanned between $\ell^1$ and $\ell^2$.
With a quasi-convex set $v$ we consider a fixed base segment $s_v$ 
and the points ($x$-coordinates) $v^i = s_v \cap \ell^i $ for $i=1,2$.

A sequence $(v_1,\ldots,v_k)$ of already presented quasi-convex sets
is \emph{$i$-in\-crea\-sing} for $i = 1,2$ if we have $v_1^i \leq v_2^i\leq\ldots \leq v_k^i$.  
The reverse of an $i$-increasing sequence is
called \emph{$i$-decreasing} for $i=1,2$.  Let $\alpha_v$ be the size
of a maximum sequence $S_\alpha(v)$ of already presented sets that is
$1$-increasing and $2$-decreasing and starts with $v$.  Let $\beta_v$
be the size of a maximum sequence $S_\beta(v)$ of already presented
sets that is $1$-decreasing and $2$-increasing and starts with $v$.
 
The algorithm is going to color $v$ with a triple $(\alpha_v, \beta_v,
\gamma_v)$ where $\alpha_v$ and $\beta_v$ are defined as above. 
The definition of  $\alpha_v$ and $\beta_v$ is as in Schmerl's 
on-line algorithm for chain partitions of 2-dimensional orders or
equivalently on-line coloring of permutation graphs. Indeed, if 
the input consists of a set of segments, then any two segments
with the same $\alpha$- and $\beta$-values are disjoint. 

For a
fixed pair $(\alpha,\beta)$ consider the set $X =
X(\alpha,\beta)$ of all quasi-convex sets $u$ presented so far 
that have been colored colored with $(\alpha,\beta,*)$, where $*$ an
arbitrary value of the third coordinate. Since $S_\alpha(v)\cup
S_\beta(v)$ is a collection of sets with pairwise intersecting 
base segments we can conclude that $\alpha_v+\beta_v= |S_\alpha(v)|+
|S_\beta(v)| = 1 + |S_\alpha(v)\cup S_\beta(v)| \leq 1+ \omega$.

To determine $\gamma_v$ the algorithm uses First-Fit on the set
$X(\alpha,\beta)$.  Bosek et al.~\cite{BKS10} have shown 
that  First-Fit is efficient on cocomparability graphs with no induced
$K_{t,t}$. The best bound is due to Dujmovi\'c, Joret and
Wood~\cite{DJW12}: First-Fit uses at most $8(2t-3)\omega$ colors
on cocomparability graphs with no induced $K_{t,t}$.

To make the result applicable we show
that the intersection graph of each class $X(\alpha,\beta)$ is a
cocomparability graph with no induced $K_{3,3}$. As the number of
these sets is at most $\binom{\omega+1}{2}$, this will conclude the
proof.
 
\begin{claim} 
The bases of sets in $X(\alpha,\beta)$ are pairwise disjoint.
\end{claim}
\begin{claimproof} 
Consider any two sets $u_1,u_2 \in X$ with the
endpoints $u_j^i \in u_j \cap \ell^i$ for $i=1,2$ and $j=1,2$ of their
bases.  It suffices to show that we have $u_1^i < u_2^i$ for $i=1,2$
or $u_1^i > u_2^i$ for $i=1,2$.
  
Assume that $u_1^1 \leq u_2^1$ and $u_1^2 \geq u_2^2$ and that $u_1$ was
presented before $u_2$.  Since $u_1 \in X(\alpha,\beta)$ it is part of a
$1$-decreasing and $2$-increasing sequence
$(u_1,v_2,\ldots,v_{\beta})$. The sequence
$(u_2,u_1,v_2\ldots,v_{\beta})$ is a longer
$1$-decreasing and $2$-increasing sequence starting with $u_2$. 
This contradicts the fact that $u_2 \in X(\alpha,\beta)$.
  
A similar argument applies when $u_2$ was presented before $u_1$.
In this case we compare the 
$1$-increasing and $2$-decreasing sequences
$(u_2,v_2,\ldots,v_\alpha)$ 
and  $(u_1,u_2,v_2,\ldots,v_\alpha)$ to arrive at a contradiction.
\end{claimproof}
 
\begin{claim} The intersection graph of $X(\alpha,\beta)$ contains no induced
$K_{3,3}$.
\end{claim}
\begin{claimproof} 
Let $U$ and $V$ be any two disjoint triples of sets in $X$.  We
shall show that if $U$ and $V$ are independent, then there is a set
in $U$ which is disjoint from a set in $V$, i.e., that the
intersection graph of these six sets is not an induced $K_{3,3}$
with bipartition classes $U, V$.
  
By the previous claim the bases of these six sets in $U \cup V$ are
disjoint and hence are naturally ordered from left to right within the
strip.  Without loss of generality amongst the leftmost three bases at
least two belong to sets in $U$ and thus amongst the rightmost 
three bases at least two belong to sets in $V$.  In particular, there are
four sets $u_1, u_2 \in U$, $v_1,v_2 \in V$ whose left to right order
of bases is $u_1,u_2,v_1,v_2$.
  
By assumption $u_1,u_2$ and $v_1,v_2$ are non-intersecting.  Since the
base of each set is contained in the corresponding set
(quasi-convexity) we know that $u_1$ lies completely to the left of
the base of $u_2$ and $v_2$ lies completely to the right of the base
of $v_1$. Together with the order of the bases of $u_2$ and $v_1$ this
makes $u_1$ and $v_2$ disjoint.
\end{claimproof}

\end{proof}

It is possible to decrease the number of colors used by the algorithm
from $\binom{\omega+1}{2}\cdot24\omega$ to $\binom{\omega+1}{2}\cdot
16\omega$ by showing that the pathwidth of the intersection graph of
$X(\alpha,\beta)$ is at most $2\omega-1$ and applying another
result from~\cite{DJW12}: First-Fit on cocomparability graph of
pathwidth at most $t$ uses at most $8(t+1)$ colors.

\begin{proof}[Proof of Proposition~\ref{prop:convex-vs-dim-and-height}]
  Let $P$ be any poset of height $2$, and let $X$ and $Y$ be the sets
  of minimal and maximal elements in $P$, respectively.  We represent
  the elements in $Y$ as pairwise intersecting segments so that every
  segment appears on the left envelope, that is, on every segment $y
  \in Y$ there is a point $r_y$ such that the horizontal ray emanating
  from $r_y$ to the left has no further intersection with segments
  from $Y$.  Choose $p \in \ell^1$ and $q \in \ell^2$ to the left of
  all segments for $Y$ and define for each $x \in X$ the convex set
  $C_x$ as the convex hull of $p,q$ and the set of all $r_y$ for which
  $y$ and $x$ are incomparable in $P$.  It is easy to check that in
  the resulting representation two sets intersect if and only if the
  corresponding elements in $P$ are incomparable. See
  Figure~\ref{fig:height-2-convex} for an illustration.
 
 \begin{figure}[htb] \centering
  \includegraphics{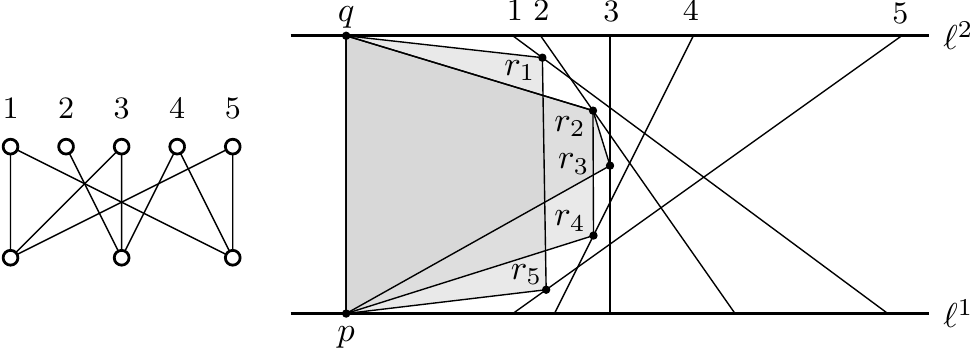}
  \caption{A poset of height $2$ and its representation with convex
sets spanned between two lines.}
  \label{fig:height-2-convex}
 \end{figure}

\begin{figure}[htb] \centering
  \includegraphics{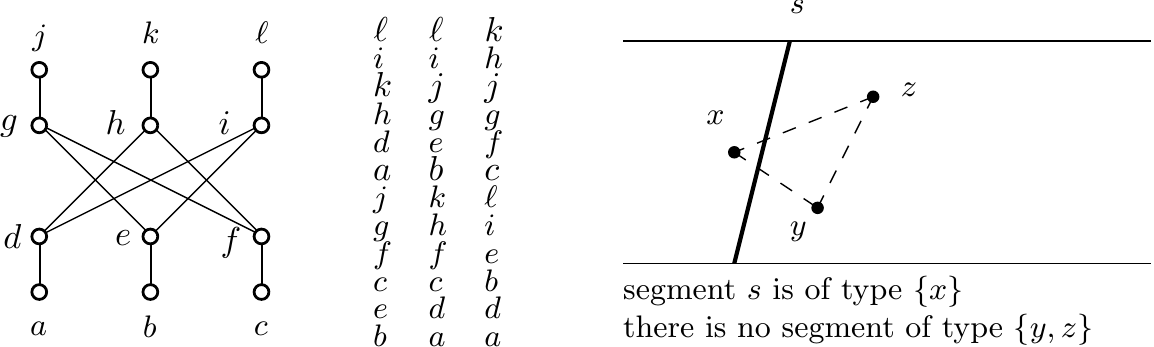}
  \caption{A height $3$ and $3$-dimensional poset that is not
quasi-convex. Provided with its $3$ linear extensions witnessing the
dimension.}
  \label{fig:height-3}
\medskip
\end{figure} 
 
We claim that the poset $Q$ depicted in Figure~\ref{fig:height-3} is
not quasi-convex. Suppose that there is a quasi-convex realization of
$Q$.  Fix three points within the strip $x\in a\cap j$, $y\in b\cap k$, and $z\in c\cap
\ell$.  The \emph{type} of a segment $s$ spanned in the strip and
avoiding $x$, $y$ and $z$ is the subset of $\set{x,y,z}$ consisting of
the points that are to the left of $s$.  How many different types of
segments can exist for given $x$, $y$ and $z$?  We claim that among
$8$ possible subsets only $7$ are realizable.  Indeed, consider the
point $p\in\set{x,y,z}$ with the middle value with respect to the
vertical axis.  Then either $\set{p}$ or $\set{x,y,z}\setminus\set{p}$
is not realizable (see Figure~\ref{fig:height-3}).  A collection of
quasi-convex sets representing the elements $d,e,f,g,h,i$ of $Q$ must
have base segments of pairwise distinct types. Moreover the types
$\emptyset$ and $\set{x,y,z}$ do not occur.  This leaves $5$ possible
types for $6$ elements, contradiction.
\end{proof}

\begin{proposition}\label{prop:not-convex}
 There is a quasi-convex poset that is not convex.
\end{proposition}
\begin{proof} Consider the poset $Q$ on the set $E =
\{a,b,c,d,e,f\}$ as shown in Figure~\ref{fig:not-convex-poset}.  Moreover,
consider the representation $\calR$ of $Q$ with segments spanned
between two lines given in Figure~\ref{fig:not-convex-representation}.
Each cell $w$ in $\calR$ naturally corresponds to the downset $D_w$ in
$Q$ (downwards closed subset of $E$) formed by those segments in
$\calR$ that lie completely to the left of $w$.
 
 We shall construct a quasi-convex poset $\bar{Q}$ that has $Q$ as an
induced sub-poset as follows.  For each cell $w$ in $\calR$
corresponding to a downset $D_w \subseteq E$ of~$Q$ there are two
incomparable elements $w_1,w_2$ in $\bar{Q}$, where $w_2$ is above all
elements in $D_w$ and $w_1$ is below all elements in $E \setminus D_w$.
There are no further comparabilities between $w_1,w_2$ and elements of
$\bar{Q}$, except for those implied by transitivity.  We refer again
to Figure~\ref{fig:not-convex} for an illustration.

 \begin{figure}[htb] \centering
\subfigure[\label{fig:not-convex-poset}]{
   \includegraphics{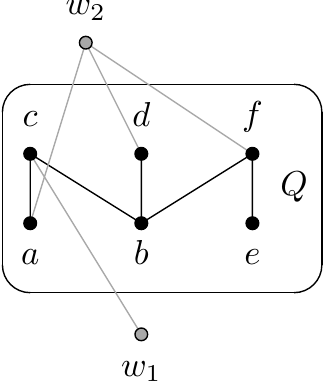} } \hfill
\subfigure[\label{fig:not-convex-representation}]{
   \includegraphics{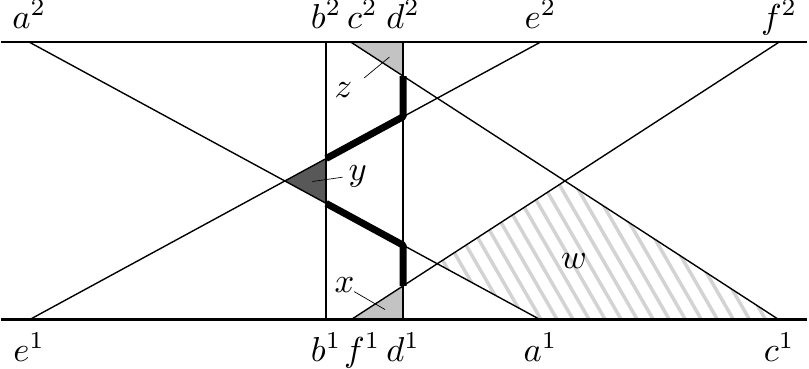} }
  \caption{\subref{fig:not-convex-poset} The $6$-element poset $Q$ and
the elements $x_1,x_2$ of $P$ for the downset $D = \{a,b,d,e,f\}$ in
$Q$. \subref{fig:not-convex-representation} The segment representation
$\calR$ of $Q$ with the cells $w$, $x$, $y$ and $z$ corresponding to
the downsets $D$, $\{b,e,f\}$, $\{a,e\}$ and $\{a,b,c\}$ in $Q$,
respectively.}
  \label{fig:not-convex}
\medskip
 \end{figure}

 Fix any quasi-convex representation $\bar{\calR}$ of $\bar{Q}$.  By
definition each quasi-convex set in $E$ comes with a base segment spanned
between the two lines.  With $\calR'$ we denote the configuration of
the base segments corresponding to elements of $Q$.

\begin{claim} 
  The segment representations $\calR$ and $\calR'$ are equivalent
  in the sense that the segments together with $\ell^1$ and $\ell^2$ induce
  (up to reflection) the same plane graph where vertices are
  attachment points and crossings of segments and edges are pieces of
  segments/lines between consecutive vertices).
 \end{claim}
 \begin{claimproof} Consider any cell $w$ in $\calR$ and the
corresponding downset $D_w \subseteq E$ of $Q$.  By the definition of
$w_1,w_2$ in $\bar{Q}$ (in particular the fact that the corresponding
sets intersect in $\bar{\calR}$), there is a cell $\bar{w}$ in
$\calR'$ that lies to the right of all sets in $D_w$ and to the left of
all sets in $E \setminus D_w$.  Since there can be only one such cell in
$\calR'$, we have an injection $\varphi$ from the cells of $\calR$
into the cells of $\calR'$.
  
Next note that if $s,t$ are two intersecting segments in $\calR$, then
there is a cell in $\calR$ with $s$ to the left and $t$ to the right,
as well as another cell with $s$ to the right and $t$ to the left.
With $\varphi$ we have such cells also in $\calR'$ and hence the
segments for $s$ and $t$ in $\calR'$ intersect as well.  Disjoint
segments in $\calR$ represent a comparability in $Q$, hence, the
corresponding segments in $\calR'$ have to be disjoint as well. It
follows that the number of intersections and, hence, also the number of
cells, is the same in $\calR$ and $\calR'$, proving that $\varphi$ is
a bijection.

To show that $\calR$ and $\calR'$ are equivalent we now consider the
dual graphs. That is we take the cells between the lines as vertices
and make them adjacent if and only if the corresponding downsets
differ in exactly one element.  These dual graphs come with a plane
embedding. All the inner faces of these embeddings correspond to
crossings and are therefore of degree $4$.  Moreover, every $4$-cycle
of these graphs has to be an inner face.  This uniquely determines (up
to reflection) the embeddings of these dual graphs and hence also of
the primal graphs. For the last conclusion we have used that the union
of all segments in $\calR$ and $\calR'$ is connected.
 \end{claimproof}

We extend $\bar{Q}$ by adding an element $g$ below $d$ and
$y_2$, but incomparable to $x_2$ and $z_2$, where $y$, $x$ and $z$ are
the cells in $\calR$ corresponding to downsets $\{a,e\}$, $\{b,e,f\}$,
$\{a,b,c\}$ in $Q$, respectively
(Figure~\ref{fig:not-convex-representation}).
Let $P$ be the poset after adding $g$.
 
\begin{claim} 
$P$ is not convex.
\end{claim}
\begin{claimproof} 
By the previous claim every quasi-convex representation of $P$
induces a segment representation $\calR'$ of $Q$ equivalent to
$\calR$.  We denote the segments in $\calR'$ for elements
$a,b,c,d,e,f$ by $a^*,b^*,c^*,d^*,e^*,f^*$, respectively, and the
cells in $\calR'$ corresponding to $x,y,z$ in $\calR$ by
$x^*,y^*,z^*$, respectively.  We claim that $x^*$ lies strictly
below $y^*$, which lies strictly below $z^*$.  Indeed, we can
construct a $y$-monotone curve as follows
(Figure~\ref{fig:not-convex-representation}): Start with the highest
point of $x^*$, i.e., the crossing of $f^*$ and $d^*$, follow $d^*$
to its crossing with $a^*$, follow $a^*$ to its crossing with $b^*$,
i.e., the lowest point of the cell $y^*$.  And symmetrically, we go
from the lowest point of $y^*$ (the crossing of $b^*$ and $e^*$)
along $e^*$ to its crossing with $d^*$ and along $d^*$ to its
crossing with $c^*$, i.e., the highest point of $z^*$.
  
Now, as $g$ is below $d$, but incomparable to $x_2$, the set for $g$
contains a point $p$ right of $f^*$ and left of $d^*$, i.e., $p \in
x$. Similarly, the set for $g$ contains a point $q \in z$.
Moreover the segment between $p$ and $q$ lies between the segments
$b^*$ and $d^*$ as it starts and ends there.  However, the base
segment for $y_2$ lies to the right of $d^*$ as $y_2$ is to the right
of $e$ and $a$.  Hence, if $g$ were a convex set, then the sets $g$
and $y_2$ would intersect, contradicting that $g$ is below $y_2$ in
$P$.
\end{claimproof}
 
To see that $P$ is quasi-convex take the representation $\calR$ of
$Q$, select a point $p_w$ in each cell $w$, let $s$ and $t$ be two
segments such that $s$ is on the left and $t$ is on the right of all
segments in $\calR$. For each cell $w$ of $\calR$ define T-shaped sets
for $w_1$ and $w_2$ consisting of $s$ and $t$, respectively, together
with a horizontal segment ending at $p_w$.  Finally let $g$ be the
union of $s$ and two horizontal segments, one ending at $p_x$ and one
at $p_z$.
\end{proof}

\section{On-line Curve Representation}\label{sec:connected-btw-two-lines}

In this section we prove Theorem~\ref{thm:online-curves}, i.e., we
show that there is an on-line algorithm that produces a curve
representation of any poset that is given on-line. The curves 
used for the representation are $y$-monotone.

Recall that a \emph{linear extension $L$} of a poset $P$ is a total ordering of
its elements such that if $x < y$ in $P$ this implies $x < y$ in $L$. 
Our construction maintains the invariant that at all times the curve representation
$\calC$ of the current poset $P$ satisfies the following
property~\eqref{eq:*}:
 \begin{equation}\label{eq:*}\tag{$*$}
  \begin{minipage}{0.92\linewidth}
    there is a set $\calL$ of horizontal lines such that for every
    linear extension $L$ of $P$ there is a horizontal line $\ell \in
    \calL$ such that the curves in $\calC$ intersect
    $\ell$ from left to right in distinct points in the order given by
    $L$.
  \end{minipage}
 \end{equation} 

 For the first element of the poset use any vertical segment in the
 strip and property~\eqref{eq:*} is satisfied.  Assume that for the
 current poset $P$ we have a curve representation $\calC$ with
 $y$-monotone curves respecting~\eqref{eq:*}.  
 
 Let $x$ be a new element extending $P$.  
 The elements of $P$
 are partitioned into the upset $U(x) = \{y : x<y\}$, the downset
 $D(x)=\{y : y < x\}$, and the set $I(x) = \{y : x||y \}$ of
 incomparable elements. Let $S$ be the union of all points in the
 strip between $\ell^1$ and $\ell^2$ that lie strictly to the left of
 all curves in $U(x)$ and strictly to the right of all curves in
 $D(x)$.  Note that $S$ is $y$-monotone (its intersection with any
 horizontal line is connected), $S \cap \ell^i \neq \emptyset$ for
 $i=1,2$, and that $S$ is connected since each curve in $U(x)$ lies
 completely to the right of each curve in $D(x)$.  This implies that
 for any two points $p,q \in S$ with distinct $y$-coordinates there is
 a $y$-monotone curve connecting $p$ and $q$ inside of $S$.

We use the set $\calL$ to draw the curve for $x$ as follows:

\begin{itemize}
\item 
  Choose $\varepsilon > 0$ small enough so that within the
  $\varepsilon$-tube $\ell_\varepsilon$ of any line $\ell \in \calL$
  no two curves get closer than $\varepsilon$.
\item 
  For each line $\ell \in \calL$ choose two points $q_\ell, p_\ell \in
  \ell_\varepsilon \cap S$ such that $q_\ell$ is above $\ell$ and 
  has distance at most $\varepsilon$ to the left boundary of $S$ while
  $p_\ell$ is below $\ell$ and 
  has distance at most $\varepsilon$ to the right boundary of $S$. 
  Draw a segment from $p_\ell$ to $q_\ell$.
\item If $\ell$ and $\ell'$ are consecutive in $\calL$ with $\ell$
  below $\ell'$, then we connect $q_\ell$ and $p_{\ell'}$ 
  by a $y$-monotone curve in $S$. We also connect
  the lowest $p$ and the highest $q$ by  $y$-monotone curves in
  $S$ to $\ell^1$ and  $\ell^2$ respectively.
\item 
  The curve of $x$ is the union of the segments $\overline{p_\ell q_\ell}$ and the
  connecting curves.
\end{itemize}
 
Figure~\ref{fig:online-curve-representation} illustrates the construction.
 
 \begin{figure}[htb]
  \centering
  \includegraphics[scale=.75]{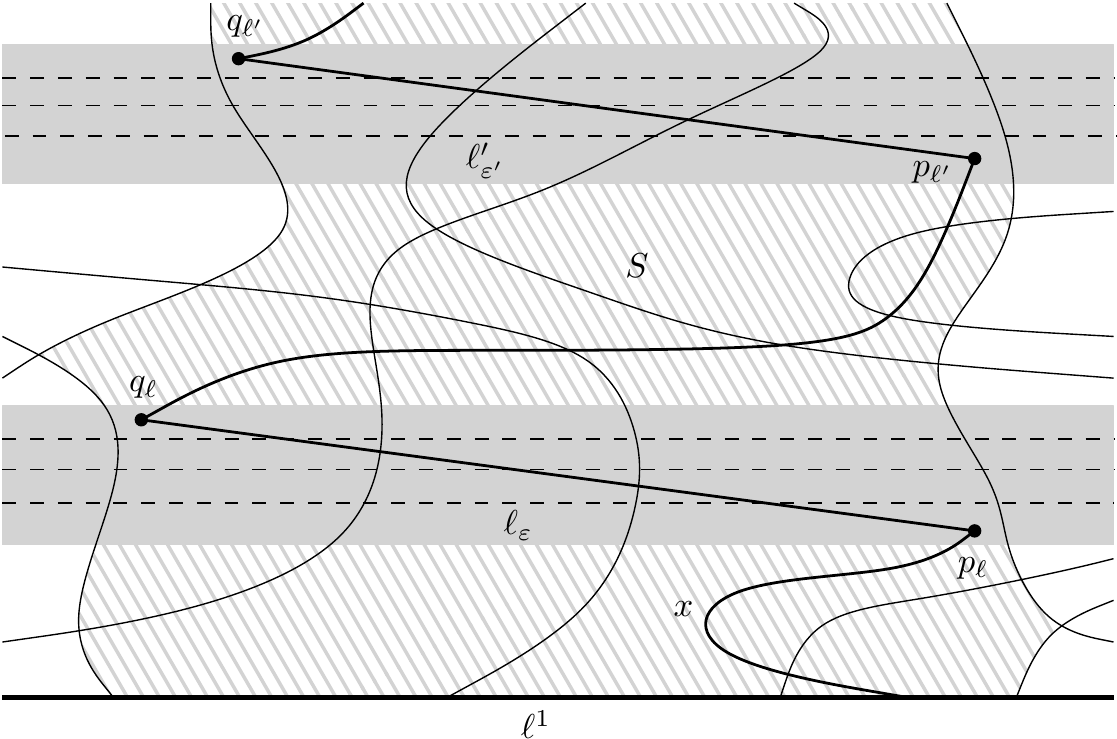}
  \caption{Constructing the curve for a new element $x$ by
    using segments within $\varepsilon$-tubes for each $\ell \in
    \calL$. Dashed horizontal lines correspond to the lines in $\calL_{P \cup \{x\}}$.}
  \label{fig:online-curve-representation}
\medskip
\end{figure}

We claim that the curve representation of $P$ together with the curve
of $x$ has property~\eqref{eq:*}.  Let $L = (\ldots,a,x,b,\ldots)$ be
an arbitrary linear extension of $P \cup \set{x}$ and let $L^x =
(\ldots,a,b,\ldots)$ be the linear extension of $P$ obtained from $L$
by omitting $x$.  Let $\ell^x \in \calL$ be the horizontal line
corresponding to $L^x$.  Within the $\varepsilon$-tube of $\ell^x$ the
segment $\overline{p_{\ell^x} q_{\ell^x}}$ contains a subsegment
$\overline{\rho_a\rho_b}$ where $\rho_a$ is a point $\varepsilon$
to the right of the curve of $a$ and $\rho_b$ is a point
$\varepsilon$ to the left of the curve of $b$.  The horizontal line
$\ell$ containing the point $\frac{\rho_a+\rho_b}{2}$ is a line
representing $L$ in $P\cup\set{x}$.  This proves property~\eqref{eq:*} for the
extended collection of curves.

The comparabilities in the intersection of all linear extensions of $P
\cup \set{x}$ are exactly the comparabilities of $P \cup
\set{x}$. Therefore, property~\eqref{eq:*} implies that the curve of
$x$ is intersecting the curves of all elements of $I(x)$. Since the
curve of $x$ is in the region $S$ it is to the right of all curves in
$D(x)$ and to the left of all curves in $U(x)$. Hence, the extended
family of curves represents $P \cup \set{x}$.

\section{Connected Sets Attached to a Line}
\label{sec:attached}

In this section we give the proof of Proposition~\ref{prop:attached-impossible}.
Actually, we prove a stronger statement by induction:

\begin{claim}
  The adversary has a strategy $S_k$ to create a family of segments
  attached to a horizontal line $h$ with clique number at most $2$
  against any on-line coloring algorithm $A$ such that there is a
  vertical line $v$
  with the properties:
 \begin{enumerate}
 \item any two segments pierced by $v$ are disjoint,
 \item every segment pierced by $v$ is attached to $h$ to the right of $v$,
  \item $A$ uses at least $k$ distinct colors on segments pierced by $v$.
 \end{enumerate}
\end{claim}
\begin{claimproof}
  The strategy $S_1$ only requires a single segment with negative
  slope.  Now consider $k \geq 2$.  Fix any on-line algorithm $A$.
  The strategy $S_{k}$ goes as follows.  First the adversary uses
  $S_{k-1}$ to create a family of segments $\calF_1$ and a vertical
  line $v_1$ piercing a set $V_1\subseteq \calF_1$ of pairwise
  disjoint segments on which $A$ uses at least $k-1$ colors.  Define a
  rectangle $R$ with bottom-side on $h$, the left-side in $v_1$ 
  and small enough such that the vertical line supported by the right-side is
  piercing the same subset $V_1$ of $\calF_1$, moreover $R$ is
  disjoint from all the segments in $\calF_1$.  The adversary uses
  strategy $S_{k-1}$ again, this time with the restriction that all
  the segments are contained in $R$.  This creates a family $\calF_2$
  and a vertical line $v_2$ piercing a set $V_2\subseteq \calF_2$ of
  pairwise disjoint segments on which $A$ uses at least $k-1$ colors.
  By construction segments from $\calF_1$ and $\calF_2$ are pairwise
  disjoint. From the definition of $R$ it follows that line $v_2$
  intersects all the segments in $V_1$ and no other segments from
  $\calF_1$.  Strategy $S_k$ is completed with the creation of one
  additional segment $d$ such that $d$ is
  attached between $v_1$ and $v_2$, $d$ is intersecting all the
  segments in $V_2$ and the vertical line
  $v_1$ but it intersects none of the segments in $V_1$
  (see~Figure~\ref{fig:attached-impossible}).

 \begin{figure}[htb]
  \centering
  \includegraphics{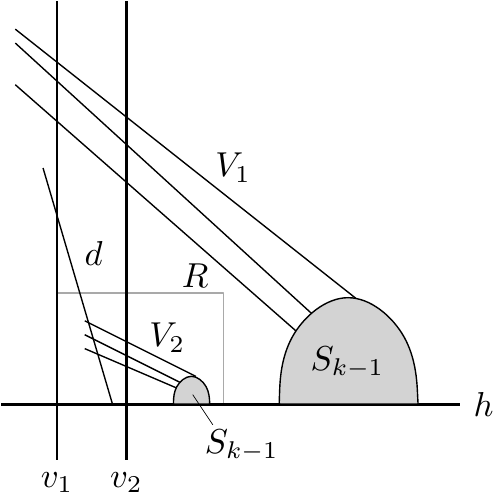}
  \caption{Strategy $S_k$ consists of two calls of strategy
    $S_{k-1}$ and an addition of an extra segment $d$. Algorithm $A$
    unavoidably uses $k$ colors on the segments intersecting $v_1$ or
    on the segments intersecting $v_2$. 
  \label{fig:attached-impossible}}
 \end{figure}

 If $A$ uses at least $k$ distinct colors on $V_1\cup V_2$ then $v_2$
 is the vertical line witnessing the invariant.  Otherwise $A$ uses
 exactly the same set of $k-1$ colors on $V_1$ and $V_2$ and since
 segment $d$ intersects all segments in $V_1$ it must be colored with
 a different color.  Thus, the vertical line $v_1$ intersecting
 $V_1\cup\set{d}$ intersects segments of at least $k$ distinct colors.
\end{claimproof}

\section{Open problems}

In this concluding section we collect some open problems related to
the results of this paper. 
 
In Figure~\ref{fig:poset-classes} there are some classes of posets
that contain interval orders and 2-dimensional orders and are
contained in the class of convex orders. For on-line coloring of the
cocomparability graphs of these classes (given with a representation)
we have the algorithm from Theorem~\ref{thm:main-convex} that uses
$O(\omega^3)$ colors.

\Item{(1)} Find an on-line algorithm that only needs $O(\omega^\tau)$
($\tau <3$) colors for coloring graphs in a class $\calG$ between
2-dimensional and convex. Interesting choices for $\calG$ would be 
trapezoid graphs, bounded tolerance graphs, triangle graphs
(or simple triangle graphs; for the definition cf.~\cite{Mer13}).
\medskip

\noindent
By restricting the curves or the intersection pattern of curves 
spanned between two lines we obtain further classes of orders 
which are nested between 2-dimensional orders and the class of all
orders. We define {\em $k$-bend orders} by restricting the number of
bends of the polygonal curves representing the elements to $k$.
Clearly, every $k+2$ dimensional order is a $k$-bend order.
We define {\em $k$-simple orders} by restricting the number of
intersections of pairs of curves representing elements of the order to $k$.

\Item{(2)} Find on-line algorithms that only need polynomially
many colors for coloring cocomparability graphs 
of $2$-simple or $1$-bend orders when a representation is given.
\medskip

\noindent
Another direction would be the study of recognition complexity.
Meanwhile the recognition complexity for all classes shown in
Figure~\ref{fig:poset-classes}, except convex orders, has been
determined (see~\cite{Mer13}). 

\Item{(3)} Determine the recognition complexity for convex orders.
\medskip

\noindent
We think that the determination of the recognition complexity 
of $2$-simple orders and $1$-bend orders are also interesting problems.

\bibliographystyle{plain}
\bibliography{paper}
\end{document}